\documentclass[12pt]{amsart}
\usepackage{times,fullpage}
\usepackage{latexsym,amscd,amssymb,url}
\newtheorem{thm}{Theorem}
\newtheorem{prop}{Proposition}
\newtheorem{lem}{Lemma}
\newtheorem{cor}{Corollary}

\newtheorem{prob*}{Problem}

\theoremstyle{remark}
\newtheorem{rem}{Remark}
\newtheorem{ex}{Example}

\theoremstyle{definition}

\newcommand{\Om}{\Omega^{U}}

\newcommand{\C}{\mathbb{ C}}

\newcommand{\Hom}{\operatorname{Hom}}

\newcommand{\II}{\mathcal{ I}}

\newcommand{\Q}{\mathbb{ Q}}
\newcommand{\R}{\mathbb{ R}}
\newcommand{\rk}{\operatorname{rk}}

\newcommand{\Z}{\mathbb{ Z}}

\newcommand{\Hi}{{\mathcal Hir}}
\newcommand{\Ho}{\mathcal{ H}}
\newcommand{\Po}{\mathcal P}

\newcommand{\CH}{\mathcal{ CH}}
\newcommand{\PP}{\mathcal{ PP}}
\newcommand{\J}{{\mathcal J}}
\newcommand{\JJ}{\mathcal{JO}}

\newcommand{\I}{{\mathcal I}}

\newcommand{\im}{\operatorname{Im}}
\newcommand{\Td}{\operatorname{Td}}

\title{The Hodge ring of K\"ahler manifolds}
\author{D.~Kotschick}
\address{Mathematisches Institut, {\smaller LMU} M\"unchen,
Theresienstr.~39, 80333~M\"unchen, Germany}
\email{dieter@member.ams.org}
\author{S.~Schreieder}
\address{Mathematisches Institut, {\smaller LMU} M\"unchen,
Theresienstr.~39, 80333~M\"unchen, Germany; and Trinity College, Cambridge, CB2 1TQ, United Kingdom}
\email{stefan.schreieder@googlemail.com}
\curraddr{Max-Planck-Institut f{\"u}r Mathematik, Vivatsgasse 7, 53111 Bonn, Germany}
\date{October 8, 2012; \copyright{\ D.~Kotschick and S.~Schreieder 2011}}
\thanks{We are very grateful to J.~Bowden for a stimulating question.}

\begin{document}

\maketitle

\centerline{{\it Dedicated to the memory of F.~Hirzebruch}}

\begin{abstract}
We determine the structure of the Hodge ring, a natural object encoding the Hodge numbers of all compact K\"ahler manifolds.
As a consequence of this structure, there are no unexpected relations among the Hodge numbers, and no essential differences
between the Hodge numbers of smooth complex projective varieties and those of arbitrary K\"ahler manifolds. 
The consideration of certain natural ideals in the Hodge ring allows us to determine exactly
which linear combinations of Hodge numbers are birationally invariant, and which are topological invariants.
Combining the Hodge and unitary bordism rings, we are also able to treat linear combinations of Hodge and 
Chern numbers. In particular, this leads to a complete solution of a classical problem of Hirzebruch's.
\end{abstract}

\section{Introduction}

For the purpose of studying the spread and potential universal relations among the Betti numbers of manifolds,
one can use elementary topological operations such as connected sums to modify the Betti numbers in examples.
This leads to the conclusion that there are no universal relations among the Betti numbers, other than the ones 
imposed by Poincar\'e duality. However, not every set of Betti numbers compatible with Poincar\'e duality is
actually realized by a (connected) manifold. This subtlety is removed, and the discussion in different dimensions
combined into one, by the following definition: consider the Betti numbers as $\Z$-linear functionals on formal 
$\Z$-linear combinations of oriented equidimensional manifolds, and identify two such linear combinations if they
have the same Betti numbers and dimensions. The quotient is a graded ring, the oriented Poincar\'e ring $\Po_*$, 
graded by the dimension, with multiplication induced by the Cartesian product of manifolds. This ring has an interesting
structure, which we determine in Section~\ref{s:Pring} below. It turns out that $\Po_*$ is finitely generated by manifolds 
of dimension at most $4$, but is not a polynomial ring over $\Z$, although it does become a polynomial ring after 
tensoring with $\Q$. 

In Section~\ref{s:rings} we carry out an analogous study for the Hodge numbers of compact K\"ahler manifolds.
This is potentially much harder, since there is no connected sum or similar cut-and-paste operation in the K\"ahler 
category that would allow one to manipulate individual Hodge numbers $h^{p,q}$ in examples. Indeed, it seems to 
have been unknown until now, whether there are any universal relations among the Hodge numbers of K\"ahler manifolds 
beyond the symmetries $h^{q,p}=h^{p,q}=h^{n-p,n-q}$.
Complex algebraic geometry does provide many constructions of K\"ahler manifolds, but these constructions are not
as flexible as one might want them to be. Moreover, in spite of the recent work of Voisin~\cite{V}, the gap between 
complex projective varieties on the one hand and compact K\"ahler manifolds on the other is far from understood.
We refer the reader to Simpson's thought-provoking survey~\cite{Si} for a description of the general state of ignorance concerning
the spread of Hodge numbers and other invariants of K\"ahler manifolds.

It is our goal here to shed some light on the behaviour and properties of Hodge numbers of K\"ahler manifolds. For this 
purpose we consider the Hodge numbers as $\Z$-linear functionals on formal $\Z$-linear combinations of compact 
equidimensional K\"ahler manifolds and identify two such linear combinations if they have the same Hodge numbers and 
dimensions. The quotient is a graded ring, the Hodge ring $\Ho_*$, graded by 
the complex dimension, with multiplication again induced by the Cartesian product.
 Its structure is described by the following result.
\begin{thm}\label{t:Hstructure}
The Hodge ring $\Ho_*$ is a polynomial ring over $\Z$, with two generators in degree one, and one in degree two.
For the generators one may take the projective line $L=\C P^1$, an elliptic curve $E$, and any K\"ahler surface $S$
with signature $\pm 1$.
\end{thm}
Note that a priori it is not at all obvious that $\Ho_*$ is finitely generated, let alone generated by elements of small degree.
Moreover, in the topological situation of the Poincar\'e ring, the corresponding structure is more complicated,
in that $\Po_*$ is not a polynomial ring over $\Z$.

The proof of this theorem has several important consequences, including the following:
\begin{enumerate}
\item Since we may take the surface $S$ to be projective, the Hodge ring is generated by projective varieties. 
This is in contrast with the work of Voisin~\cite{V} on the Kodaira problem, which showed that more subtle features of Hodge 
theory do distinguish the topological types of projective manifolds from those of arbitrary K\"ahler manifolds.
\item Counting monomials, we see that the degree $n$ part $\Ho_n$ of the Hodge ring is a free $\Z$-module of rank equal 
to the number of Hodge numbers modulo the K\"ahler symmetries $h^{q,p}=h^{p,q}=h^{n-p,n-q}$.
Thus there are no universal $\Q$-linear relations between the Hodge numbers, other than 
the ones forced by the known symmetries.
\item The proof of Theorem~\ref{t:Hstructure} will show that the Hodge numbers $h^{p,q}$ with $0\leq q\leq p\leq n$ and 
$p+q\leq n$ form a $\Z$-module basis for $\Hom(\Ho_n,\Z)$.
Therefore there are no non-trivial universal congruences among these Hodge numbers.
\end{enumerate}
For technical reasons, we find it more convenient to work with a different definition of $\Ho_*$, rather than the one given above. 
However, it will follow from the discussion in Section~\ref{s:rings} below that the two definitions give the same result, and this fact
will establish statement (3), cf.~Remark~\ref{r:definitions}.

In working with Hodge numbers, the Hodge ring plays a r\^ole analogous to that of the unitary bordism ring $\Om_*$
in working with Chern numbers. This bordism ring is also generated by smooth complex  projective varieties, and its 
structure shows that there are no universal $\Q$-linear relations between the Chern numbers, cf.~Subsection~\ref{ss:bordism} 
below.
However, in that case the analogue of statement (3) above is not true, in that there are universal congruences 
between the Chern numbers.

Our determination of the Hodge ring over $\Z$ allows us to write down all universal linear relations or congruences 
between the Hodge numbers of smooth projective varieties and their Pontryagin or Chern numbers:
\begin{enumerate}
\item[(HP)] A combination of Hodge numbers equals a combination of Pontryagin numbers if and only if it is a multiple of the 
signature, see Corollary~\ref{c:HP}.
\item[(HC)] A combination of Hodge numbers equals a combination of Chern numbers if and only if it is a combination of
the $\chi_p=\sum_q (-1)^q h^{p,q}$, see Corollary~\ref{c:univ}.
\end{enumerate}
In these statements the Hodge numbers are considered modulo the K\"ahler symmetries. We prove (HP) and (HC) 
in the strongest form possible, for equalities mod $m$ for all $m$; the statements over $\Z$ or $\Q$ follow. 
While the validity of these relations is of course well-known, their uniqueness is new, except that, with coefficients in $\Q$,
statement (HC) could be deduced from~\cite[Corollary~5]{Chern}. Just as it was unknown until now whether there are 
universal relations between the Hodge numbers -- we prove that there are none beyond the K\"ahler symmetries -- their 
potential relations with the Chern and Pontryagin numbers were unknown.

In Section~\ref{s:comp} we analyze the comparison map $f\colon\Ho_*\longrightarrow\Po_*$, whose image is 
naturally the Poincar\'e ring of K\"ahler manifolds. We will see  that there are no universal relations between the Betti 
numbers of K\"ahler manifolds, other than the vanishing mod $2$ of the odd-degree Betti numbers. Setting aside these
trivial congruences, the only relations between the Betti numbers of smooth projective varieties and their Pontryagin 
or Chern numbers are the following:
\begin{enumerate}
\item[(BC)] A combination of Betti numbers equals a combination of Chern numbers if and only if it is a multiple of the 
Euler characteristic, see Corollary~\ref{c:BC}.
\item[(BP)] Any congruence between a $\Z$-linear combination of Betti numbers of smooth complex projective varieties 
of complex dimension $2n$ and a non-trivial combination of Pontryagin numbers is a consequence of $e\equiv (-1)^n\sigma \mod 4$,
see Corollary~\ref{c:BPKaehler}. Here $e$ and $\sigma$ denote the Euler characteristic and the signature respectively.
\end{enumerate}
In both statements the Betti numbers are considered modulo the symmetry imposed by Poincar\'e duality. In (BP) the conclusion is 
that there are no universal $\Q$-linear relations.

We shall determine several geometrically interesting ideals in the Hodge ring. 
An easy one to understand is the ideal generated by differences of birational smooth projective 
varieties. This leads to the following result, again modulo the K\"ahler symmetries of Hodge numbers:
\begin{thm}\label{t:Hbirat}
The mod $m$ reduction of an integral linear combination of Hodge numbers is a birational invariant of projective 
varieties if and only if the linear combination is congruent modulo $m$ to a linear combination of the $h^{0,q}$.
\end{thm}
It follows that a rational linear combination of Hodge numbers is a birational invariant of smooth complex projective 
varieties if and only if, modulo the K\"ahler symmetries, it is a combination of the $h^{0,q}$ only.

Other ideals in $\Ho_*$ we will calculate are those of differences of homeomorphic or diffeomorphic
complex projective varieties, thereby determining exactly which linear combinations of Hodge numbers
are topological invariants. The question of the topological invariance of Hodge numbers was first raised 
by Hirzebruch in 1954. His problem list~\cite{Hir1} contains the following question about the Hodge and Chern numbers of 
smooth complex projective varieties, listed there as Problem~31:

{\it
Are the $h^{p,q}$ and the Chern characteristic numbers of an algebraic variety $V_n$ topological invariants of $V_n$? If not,
determine all those linear combinations of the $h^{p,q}$ and the Chern characteristic numbers which are topological invariants.
}

Since the time of  Hirzebruch's problem list almost sixty years ago, this and related questions have been raised repeatedly in other places,
such as a \url{mathoverflow} posting by S.~Kov\'acs in late 2010, asking whether the Hodge numbers of K\"ahler 
manifolds are diffeomorphism invariants.
The special case of Hirzebruch's question where one considers linear combinations of Chern numbers only, without the Hodge
numbers, was recently answered by the first author~\cite{PNAS,Chern}. That answer used the structure results of Milnor~\cite{M,TM}
and Novikov~\cite{NN} for the unitary bordism ring, exploiting the bordism invariance of Chern numbers. 
The Hodge numbers were not treated systematically in~\cite{PNAS,Chern} because they are not bordism invariants. However,
the results of those papers, and already of~\cite{JTop}, show that certain linear combinations of Hodge numbers that are bordism 
invariants because of the Hirzebruch--Riemann--Roch theorem are not (oriented) diffeomorphism invariants in complex dimensions 
$\geq 3$. This failure of diffeomorphism invariance of Hodge numbers, which can be traced to the fact that certain examples of pairs 
of algebraic surfaces with distinct Hodge numbers from~\cite{MAorient} become diffeomorphic after taking products with $\C P^1$, 
say, was also observed independently several years ago by F.~Campana (unpublished).

In spite of these observations, the question of determining which linear combinations of Hodge numbers are topological invariants 
was still wide open. In Section~\ref{s:HH} below we settle this question using the Hodge ring and the forgetful
comparison map $f\colon\Ho_*\longrightarrow\Po_*$. The result is:
\begin{thm}\label{t:main2}
The mod $m$ reduction of an integral linear combination of Hodge numbers of smooth complex projective varieties is
\begin{enumerate}
\item an oriented homeomorphism or diffeomorphism invariant if and only if it is congruent mod $m$ to a linear combination of 
the signature, the even-degree Betti numbers and the halves of the odd-degree Betti numbers, and 
\item an unoriented homeomorphism invariant in any dimension, or an unoriented diffeomorphism invariant in dimension $n\neq 2$,
if and only if it is congruent mod $m$ to a linear combination of the even-degree Betti numbers and the halves of the odd-degree Betti numbers.
\end{enumerate}
\end{thm}
The corresponding result for rational linear combinations follows.
Complex dimension $2$ has to be excluded when discussing diffeomorphism invariant Hodge numbers, 
since in that dimension all the Hodge numbers are linear combinations of Betti numbers and the signature, and the signature is, 
unexpectedly, invariant under all diffeomorphisms, even if they are not assumed to preserve the 
orientation, see~\cite[Theorem~6]{BLMSorient}, and also~\cite[Theorem~1]{JTop}.

In Section~\ref{s:CH} we consider arbitrary $\Z$-linear combinations of Hodge and Chern numbers.
In the same way that the Hodge numbers lead to the definition of $\Ho_*$, these more general linear
combinations lead to the definition of another ring, the Chern--Hodge ring $\CH_*$. 
We use $\CH_*$ to prove that Theorem~\ref{t:Hbirat} remains true for mixed linear combinations of Hodge and Chern numbers 
in place of just Hodge numbers. In this general setting, the conclusion of course has to be interpreted modulo the HRR relations, 
see Theorem~\ref{t:CHbiratideal} and Corollary~\ref{c:CHbirat}, which also generalize a recent theorem about Chern numbers 
proved over $\Q$ by Rosenberg~\cite[Theorem~4.2]{R}.

In Section~\ref{s:Hgen} we study certain ideals in $\CH_*\otimes\Q$, leading to the following answer to the general form of Hirzebruch's 
question, mixing the Hodge and Chern numbers in linear combinations:
\begin{thm}\label{t:main3}
A rational linear combination of Hodge and Chern numbers of smooth complex projective varieties is
\begin{enumerate}
\item an oriented homeomorphism or diffeomorphism invariant if and only if it reduces to a linear combination of the Betti and Pontryagin 
numbers after perhaps adding a suitable combination of the $\chi_p-\Td_p$, and
\item an unoriented homeomorphism invariant in any dimension, or an unoriented diffeomorphism invariant in dimension $n\neq 2$,
if and only if it reduces to a linear combination of the Betti numbers after perhaps adding a suitable combination of the $\chi_p-\Td_p$.
\end{enumerate}
\end{thm}
As always, the Hodge numbers are considered modulo the K\"ahler symmetries.
This Theorem is a common generalization of Theorem~\ref{t:main2} for the Hodge numbers and the main theorems of~\cite{PNAS,Chern} for 
the Chern numbers. Once again complex dimension $2$ has to be excluded in the statement about unoriented diffeomorphism invariants because 
the signature is a diffeomorphism invariant of algebraic surfaces by~\cite[Theorem~6]{BLMSorient}, see also~\cite[Theorem~1]{JTop}.
We do not state this theorem for congruences, since we are unable to prove it
if the modulus $m$ is divisible by $2$ or $3$; compare Remark~\ref{r:cong} in Section~\ref{s:Hgen}.

\subsection*{Dedication}

We dedicate this work to the memory of F.~Hirzebruch, who first formulated the main problems treated here and who was one of the 
principal creators of their mathematical context.
He read the preprint version of our solution, but sadly passed away before its publication in print.
We are fortunate to have been influenced by him. 

\section{The Poincar\'e ring}\label{s:Pring}

In the introduction we defined the Poincar\'e ring by taking $\Z$-linear combinations of oriented equidimensional 
manifolds, and identifying two such linear combinations if they have the same Betti numbers and dimensions. 
Elements of this ring can be identified with their Poincar\'e polynomials
$$
P_{t,z}(M) = (b_0(M)+b_1(M)\cdot t+\ldots+b_n(M)\cdot t^n)\cdot z^n \in \Z [t,z] \ ,
$$
where the $b_i(M)$ are the real Betti numbers of $M$. 
Here we augment the usual Poincar\'e polynomial using an additional variable $z$ in order to keep track of the 
dimension in linear combinations where the top-degree Betti number may well vanish. In this way we obtain
an embedding of the Poincar\'e ring into $\Z [t,z]$. This embedding preserves the grading 
given by $\deg (t)=0$ and $\deg (z)=1$.

The Betti numbers satisfy the Poincar\'e duality relations 
$$
b_i(M)=b_{n-i}(M) \  \textrm{for all} \  i \ , \  \textrm{ and} \ \ \   b_{n/2}(M)\equiv 0 \mod 2 \  \ \  \textrm{if} \ \  \ n\equiv 2 \mod 4 \ .
$$
Not every polynomial having this symmetry and satisfying the obvious constraints $b_i(M)\geq 0$ 
and $b_0(M)=1$ can be realized by a connected manifold. For example, it is known classically that 
$(1+t^k+t^{2k})z^{2k}$ cannot be realized if $k$ is not a power of $2$; cf.~\cite[Section~2]{Hunpub}.
We sidestep this issue by modifying the definition of the Poincar\'e ring in the following way,
replacing it by a potentially larger ring with a more straightforward definition. 

Let $\Po_n$ be the $\Z$-module of {\it all} formal augmented Poincar\'e polynomials 
$$
P_{t,z} = (b_0+b_1\cdot t+\ldots+b_n\cdot t^n)\cdot z^n \in \Z [t,z] \ ,
$$
satisfying the duality condition $b_i=b_{n-i}$ for all $i$ and $b_{n/2}\equiv 0 \pmod 2$ if $n\equiv 2 \pmod 4$, regardless 
of whether they can be realized by manifolds. 
One could show directly that all elements of $\Po_n$ are $\Z$-linear combinations of Poincar\'e polynomials of closed 
orientable $n$-manifolds, thereby proving that this definition of $\Po_n$ coincides with the one given in the introduction. 
We will not do this here, but will reach the same conclusion later on, see Remark~\ref{r:gen} below.

For future reference we note the following obvious statement.
\begin{lem}\label{l:basis}
The $\Z$-module $\Po_n$ is free of rank $[(n+2)/2]$, spanned by the following basis:
$$
e_k^n = (t^k+t^{n-k})z^n \quad \textrm{for} \quad 0\leq k < n/2 \ ,
$$
and, if $n$ is even, 
$$
e_{n/2}^n = t^{n/2}z^n \quad \textrm{if} \quad n\equiv 0 \pmod 4 \ ,
$$
respectively
$$
e_{n/2}^n = 2t^{n/2}z^n \quad \textrm{if} \quad n\equiv 2 \pmod 4 \ .
$$
\end{lem}

We define  the Poincar\'e ring by
$$
\Po_* =  \bigoplus_{n=0}^{\infty} \Po_n\subset\Z[t,z] \ .
$$
This is a graded ring whose addition and multiplication correspond to the disjoint union and the Cartesian 
product of manifolds, and the grading, induced by the degree in $\Z[t,z]$ with $\deg (t)=0$ and $\deg(z)=1$, 
corresponds to the dimension.

The structure of the Poincar\'e ring is completely described by the following:
\begin{thm}\label{t:Pstructure}
Let $W$, $X$, $Y$ and $Z$ have degrees $1$, $2$, $3$ and $4$ respectively. The oriented Poincar\'e ring $\Po_*$ is isomorphic, 
as a graded ring, to the quotient of the polynomial ring $\Z[W,X,Y,Z]$ by the homogeneous ideal $\II$ generated by 
$$
WX-2Y \ , \quad X^2-4Z \ , \quad XY-2WZ \ , \quad Y^2-W^2Z \ . 
$$
\end{thm}
\begin{proof}
Define a homomorphism of graded rings
$$
P\colon \Z[W,X,Y,Z]\longrightarrow\Po_*
$$
by setting
$$
P(W) = (1+ t)z \ , \quad
P(X) = 2tz^2 \ , \quad
P(Y) = (t+ t^2)z^3 \ , \quad
P(Z) = t^2z^4 \ .
$$
By definition, $P$ vanishes on $\II$, and so induces a homomorphism from the quotient $\Z[W,X,Y,Z]/\II$ to $\Po_*$. 
We will show that this induced homomorphism is an isomorphism. The first step is to prove surjectivity.
\begin{lem}\label{l:surj}
The homomorphism $P$ is surjective.
\end{lem}
\begin{proof}
If $n\equiv 0 \pmod 4$, then $e_{n/2}^n = P(Z^{n/4})$. Similarly, if $n\equiv 2 \pmod 4$, then $e_{n/2}^n = P(XZ^{(n-2)/4})$.
Thus we only have to prove that $e_k^n$ is in the image of $P$ for all $k<n/2$. We do this by induction on $n$.

It is easy to check explicitly that $P$ is surjective in degrees $\leq 4$. Therefore, for the induction we fix some $n\geq 5$,
and we assume that surjectivity of $P$ is true in all degrees $<n$.

Consider first the case when $n$ is even. Then for $k<n/2$ we have the following identity:
$$
e_k^n = e_0^{\frac{n}{2}-k} \cdot e_k^{\frac{n}{2}+k} - 2t^{\frac{n}{2}}z^n \ .
$$
By the induction hypothesis the two factors $e_0^{\frac{n}{2}-k}$ and $e_k^{\frac{n}{2}+k}$ are
in the image of $P$. Since we have already noted that $2t^{\frac{n}{2}}z^n$ is in the image of $P$, we conclude that 
$P$ is surjective in degree $n$.

Finally, assume that $n$ is odd. In this case we have
\begin{alignat*}{1}
e_k^n &= (1+t)z \cdot \big(\sum_{i=0}^{n-2k-1}(-1)^it^{k+i}\big)z^{n-1}\\
&=(1+t)z \cdot \big(\sum_{i\neq\frac{n-1}{2}-k}(-1)^it^{k+i}\big)z^{n-1}+(-1)^{\frac{n-1}{2}-k}(1+t)z \cdot t^{\frac{n-1}{2}}z^{n-1} \ .
\end{alignat*}
Here $(1+t)z=P(W)$ by definition, and the induction hypothesis tells us that 
$$
\big(\sum_{i\neq\frac{n-1}{2}-k}(-1)^it^{k+i}\big)z^{n-1}
$$
with $i$ running from $0$ to $n-2k-1$ is in the image of $P$.

On the one hand, if $n\equiv 1 \pmod 4$, then $t^{\frac{n-1}{2}}z^{n-1}=P(Z^{(n-1)/4})$. On the other hand, if $n\equiv 3 \pmod 4$, then we rewrite 
$$
(1+t)z \cdot t^{\frac{n-1}{2}}z^{n-1} = (t+t^2)z^3\cdot t^{\frac{n-3}{2}}z^{n-3} = P(YZ^{(n-3)/4}) \ .
$$
This completes the proof of surjectivity of $P$ in all degrees.
\end{proof}

The next step in the proof of the theorem is to estimate the rank of the degree $n$ part of the quotient $\Z[W,X,Y,Z]/\II$.

\begin{lem}\label{l:gen}
The degree $n$ part of the quotient $\Z[W,X,Y,Z]/\II$ is generated as a $\Z$-module by at most $[(n+2)/2]$ elements.
\end{lem}
\begin{proof}
A generating set is provided by the images of the monomials $W^iX^jY^kZ^l$ with $i+2j+3k+4l=n$. 
The relations $X^2=4Z$ and $Y^2=W^2Z$ from the definition of $\II$ mean that we only have to consider $j=0$ or $1$ and $k=0$ or $1$. 
Further, since $XY=2WZ$, we do not need any monomials where $j=k=1$. Finally, since $WX=2Y$, we may assume $i=0$ whenever
$j=1$. Thus, a generating set for the degree $n$ part of the quotient $\Z[W,X,Y,Z]/\II$ is given by the images of the 
monomials $W^iZ^l$, $XZ^l$ and $W^iYZ^l$.

Assume first that $n-2$ is not divisible by $4$. In this case there is no monomial of the form $XZ^l$ of degree $n$. The number 
of monomials of the form $W^iZ^l$ is $[(n+4)/4]$, and the number of monomials of the form $W^iYZ^l$ is $[(n+1)/4]$. The sum
of these two numbers is $[(n+2)/2]$, since we assumed that $n$ is not congruent to $2$ modulo $4$.

If $n\equiv 2 \pmod 4$, then there is exactly one monomial of the form $XZ^l$ of degree $n$, and in this case 
$$
1+\big[ \frac{n+4}{4}\big] +\big[ \frac{n+1}{4}\big] = \big[ \frac{n+2}{2}\big] \ .
$$
This completes the proof of the Lemma.
\end{proof}

To complete the proof of the Theorem, consider the homomorphism of graded rings 
$$
\Z[W,X,Y,Z]/\II\longrightarrow\Po_*
$$
induced by $P$. By Lemma~\ref{l:surj} this is surjective. Now $\Po_n$ is free of rank $[(n+2)/2]$ by Lemma~\ref{l:basis},
and the degree $n$ part of $\Z[W,X,Y,Z]/\II$, which surjects to $\Po_n$, is generated as a $\Z$-module by $[(n+2)/2]$
elements, according to Lemma~\ref{l:gen}. This is only possible if the degree $n$ part of $\Z[W,X,Y,Z]/\II$ is also free,
and the surjection is injective, and, therefore, an isomorphism.
\end{proof}

\begin{rem}\label{r:gen}
The generators $W$, $X$, $Y$ and $Z$ satisfy the following:
\begin{alignat*}{1}
P(W) &= P_{t,z}(S^1) \ ,\\
P(X) &= P_{t,z}(S^1\times S^1)-P_{t,z}(S^2) \ ,\\
P(Y) &= P_{t,z}(S^1\times S^2)-P_{t,z}(S^3) \ ,\\
P(Z) &= P_{t,z}(S^2\times S^2)-P_{t,z}(\C P^2) \ .
\end{alignat*}
This shows that the definition of the Poincar\'e ring used in this section gives the same ring as the one defined in the 
introduction. Indeed all elements of $\Po_*$ as defined here are $\Z$-linear combinations of Poincar\'e polynomials of 
closed orientable manifolds, and one can take $S^1$, $S^2$, $S^3$ and $\C P^2$ as generators. The generators 
$W$, $X$, $Y$ and $Z$ have the advantage of giving a simpler form for the relations generating the ideal $\II$.
\end{rem}

Theorem~\ref{t:Pstructure} has the following immediate implication, showing that away from the prime $2$ the oriented 
Poincar\'e ring is in fact a polynomial ring.
\begin{cor}\label{c:Ptensor}
Let $k$ be a field of characteristic $\neq 2$. Then $\Po_*\otimes k$ is isomorphic to a polynomial ring $k[W,X]$ on two 
generators of degrees $1$ and $2$ respectively. For the generators one may take $S^1$ and $S^2$.
\end{cor}
Since products of $S^1$ and $S^2$ have vanishing Pontryagin numbers, Corollary~\ref{c:Ptensor} 
implies that there are no universal $\Q$-linear relations between Betti and Pontryagin numbers. 
This result also follows, in a less direct way, from~\cite[Corollary~3]{Crelle}. 
The corresponding statement for congruences between integral linear combinations is slightly more subtle, 
and depends on the integral structure of the Poincar\'e ring.
\begin{cor}\label{c:BPoriented}
Any non-trivial congruence between an integral linear combination of Betti numbers of oriented manifolds and an integral linear 
combination of Pontryagin numbers is a multiple of the mod $2$ congruence between the Euler characteristic and the signature.
\end{cor}
Here, as always, the Betti numbers are considered modulo the symmetry induced by Poincar\'e duality. Non-trivial congruences
are those in which the two sides do not vanish separately.
\begin{proof}
A linear combination of Betti numbers of oriented $n$-manifolds that is congruent mod $m$ to a linear combination of Pontryagin numbers
corresponds to a homomorphism $\varphi\colon\Po_{n}\longrightarrow \Z_m$ that vanishes on all manifolds with zero Pontryagin numbers.
Consider the generating elements $W$, $X$, $Y$ and $Z$ of $\Po_*$ in Theorem~\ref{t:Pstructure}.
In terms of these elements, the $4$-sphere satisfies
$$
S^4=W^4-4WY+2Z \in\Po_* \ .
$$
Since any product with $W$, $X$, $Y$ or $S^4$ as a factor has vanishing Pontryagin numbers, Theorem~\ref{t:Pstructure}
together with this relation implies that the homomorphism $\varphi$ descends to the degree $n$ part of the quotient $\Z[Z]/2Z$. 
Now the mod $2$ reduction of the Euler characteristic induces an isomorphism between $\Z[Z]/2Z$ and $\Z_2[z^4]$.
Furthermore, the Euler characteristic is congruent mod $2$ to the signature, which is a linear combination of Pontryagin numbers by 
the work of Thom. This completes the proof.
\end{proof}

\begin{rem}
Proceeding as above, one can define the unoriented Poincar\'e ring using $\Z_2$-Poincar\'e polynomials of manifolds that 
are not necessarily orientable. It is easy to see that this ring is a polynomial ring over $\Z$, isomorphic to $\Z[\R P^1,\R P^2]$.
\end{rem}

\section{The Hodge ring}\label{s:rings}

To every closed K\"ahler manifold of complex dimension $n$ we associate its Hodge polynomial
$$
H_{x,y,z}(M) = \big( \sum_{p,q=0}^{n}h^{p,q}(M)\cdot x^p y^q\big)\cdot z^n \  \in \  \Z[x,y,z] \ ,
$$
where the $h^{p,q}(M)$ are the Hodge numbers satisfying the K\"ahler constraints $h^{q,p}=h^{p,q}=h^{n-p,n-q}$.
Like with the Poincar\'e polynomial, we have augmented the Hodge polynomial by the introduction of the additional
variable $z$, which ensures that the Hodge polynomial defines an embedding of the Hodge ring $\Ho_*$ defined
in the introduction into the polynomial ring $\Z[x,y,z]$. This embedding preserves the grading if we set $\deg (x)=\deg (y)=0$\and $\deg (z)=1$.

The Hodge polynomial refines the Poincar\'e polynomial in the sense that if one sets $x=y=t$ and collects terms, 
the Hodge polynomial reduces to the Poincar\'e polynomial. (At the same time one has to replace $z$ by
$z^2$ since the real dimension of a K\"ahler manifold is twice its complex dimension.)

Unlike in the definition used in the introduction, we now define $\Ho_n$ to be the $\Z$-module of {\it all} polynomials
$$
H_{x,y,z} =\big(\sum_{p,q=0}^{n} h^{p,q}\cdot x^p y^q\big)\cdot z^n \  \in \  \Z[x,y,z] 
$$
satisfying the constraints $h^{q,p}=h^{p,q}=h^{n-p,n-q}$. We will prove in Corollary~\ref{c:Hodge} below that all
elements of $\Ho_n$ are in fact $\Z$-linear combinations of Hodge polynomials of compact K\"ahler manifolds 
of complex dimension $n$, so that this definition agrees with the one in the introduction.

\begin{lem}\label{l:Hbasis}
The $\Z$-module $\Ho_n$ is free of rank $[(n+2)/2]\cdot [(n+3)/2]$.
\end{lem}
\begin{proof}
Given the constraints $h^{q,p}=h^{p,q}=h^{n-p,n-q}$, visualized in the Hodge diamond, it is straightforward to write down a module basis
for $\Ho_n$ with $[(n+2)/2]\cdot [(n+3)/2]$ elements.
\end{proof}

We define the Hodge ring by
$$
\Ho_* =  \bigoplus_{n=0}^{\infty} \Ho_n\subset\Z[x,y,z] \ .
$$
This is a commutative ring with a grading given by the degree. (Recall that the degrees or weights of $x$, $y$ and $z$ are $0$, $0$ and $1$ respectively.)
Multiplication corresponds to taking the Cartesian product of K\"ahler  manifolds, and the grading corresponds 
to the complex dimension. Its structure is completely described by the following:
\begin{thm} \label{t:structure}
Let $A$ and $B$ have degree one and $C$ have degree two.
The homomorphism
$$
H\colon \Z[A,B,C] \longrightarrow \Ho_*
$$
given by 
\[
H(A)=(1+xy)\cdot z\ \ ,\ \ H(B)=(x+y)\cdot z\ \  , \ \  H(C)=xy\cdot z^2 \ 
\] 
is an isomorphism of graded rings.
\end{thm}
This result can be proved by an argument that parallels the one we used in the proof of Theorem~\ref{t:Pstructure}. We 
give a different proof, that illustrates a somewhat different point of view.
\begin{proof}
In order to prove the injectivity of $H$, we need to show that there is no nontrivial polynomial in $A$, $B$ and $C$ which maps to zero under $H$.
Since there is always a prime number $p$, such that the mod $p$ reduction of such a polynomial is nontrivial, the injectivity of $H$ follows from 
the following stronger statement:
\begin{lem}\label{claim1}
Let $p$ be a prime number. 
The mod $p$ reduction of the map $H$
\[
\tilde H\colon\Z_p[A,B,C] \longrightarrow \Z_p[x,y,z] \ ,
\] 
given by sending $A$, $B$ and $C$ to the mod $p$ reductions of $H(A)$, $H(B)$ and $H(C)$, is injective.
\end{lem}
\begin{proof}
Suppose the contrary and let $n$ be the smallest degree in which $\tilde H$ is not injective.
Then $\ker(\tilde H)$ contains a nontrivial element of the form $C\cdot Q(A,B,C)+R(A,B)$, where $Q(A,B,C)$ and $R(A,B)$ 
are homogeneous polynomials with coefficients in $\Z_p$ of degrees $n-2$ and $n$ respectively. If we set $y=0$, we obtain 
$R(z,xz)=0$ in $\Z_p[x,z]$. Since $z$ and $xz$ are algebraically independent in $\Z_p[x,z]$, we conclude that the polynomial 
$R$ vanishes identically. Therefore, $C\cdot Q(A,B,C)\in\ker(\tilde H)$.
Since $\Z_p[x,y,z]$ is an integral domain in which $\tilde H(C)=xy\cdot z^2$ is a nontrivial element, 
we conclude that $Q(A,B,C)$ also lies in the kernel of $\tilde H$. This contradicts the minimality of $n$.
\end{proof}

It remains to prove the surjectivity of $H$. Counting the monomials in $A$, $B$, and $C$ of degree $n$ shows that the degree $n$ 
part of the graded polynomial ring $\Z[A,B,C]$ is a free $\Z$-module of rank $N=[(n+2)/2]\cdot [(n+3)/2]$.
By the injectivity of $H$, this is mapped isomorphically onto a submodule of $\Ho_n$, which by Lemma~\ref{l:Hbasis} is also a free 
$\Z$-module of rank $N$. Therefore, there are a basis $h_1,\ldots , h_N$ of $\Ho_n$ and non-zero integers $a_1,\ldots ,a_N$ such that 
$a_1h_1,\ldots, a_Nh_N$ is a basis of $\im(H)$. It remains to show that the integers $a_i$ are all equal to $\pm1$.
Suppose the contrary and let $p$ be a prime number which divides $a_i$.
Since $a_ih_i\in\im(H)$, this is the image of a polynomial $S(A,B,C)$. 
The mod $p$ reduction of $S$ must be nontrivial, since otherwise $a_ih_i/p$ would lie in the image of $H$.
However, the mod $p$ reduction of $a_ih_i$ vanishes by assumption, which is a contradiction with Lemma~\ref{claim1}.
This completes the proof of the theorem.
\end{proof}

From now on we use the isomorphism $H$ to identify $A$, $B$ and $C$ with their images in $\Ho_*$.
The following corollary paraphrases Theorem~\ref{t:Hstructure} stated in the introduction, and  
explains that instead of $A$, $B$ and $C$ one may choose different generators for $\Ho_*$.
Before we state it, note that by the Hodge index theorem the signature of manifolds induces a ring homomorphism 
$\sigma\colon\Ho_* \longrightarrow \Z[z]$, given by $x\mapsto -1$, $y\mapsto 1$.
\begin{cor}\label{c:Hodge}
Let $E$ be 
an elliptic curve, $L$ 
the projective line and let $S$ be an element in $\Ho_2$ with signature $\pm 1$. 
(For instance, $S$ might be 
a K\"ahler surface with signature $\pm 1$.)
Then, $\Ho_*$ is isomorphic to the polynomial ring $\Z[E,L,S]$. 
\end{cor}
\begin{proof}
First of all, note the identities $A=L$ and $B=E-L$, which allow us to replace the generators $A$ and $B$ in degree one by $E$ and $L$.
We may represent the element $S$ with respect to the basis $A^2$, $AB$, $B^2$ and $C$ of $\Ho_2$, given by Theorem~\ref{t:structure}.
It remains to show that in this representation, the basis element $C$ occurs with coefficient $\pm 1$.
Since $A$ and $B$ have zero signature and $C$ has signature $-1$, this is equivalent to $S$ having signature $\pm 1$, which is true by assumption. 
\end{proof}

\begin{rem}\label{r:definitions}
We have now proved that all formal Hodge polynomials are indeed $\Z$-linear combinations of Hodge polynomials of 
K\"ahler manifolds. This shows that the definition of $\Ho_*$ given at the beginning of this section gives the same ring
as the definition in the introduction, and it proves statement (3) from the introduction.
\end{rem}

The last Corollary also leads to the following result, which generalizes~\cite[Theorem~6]{PNAS}, proved there rather indirectly.
\begin{cor}\label{c:HP}
The mod $m$ reduction of a $\Z$-linear combination of Hodge numbers equals the mod $m$ reduction of a linear combination 
of Pontryagin numbers if and only if, modulo $m$, it is a multiple of the signature.
\end{cor}
\begin{proof}
If in complex dimension $2n$, a $\Z$-linear combination of Pontryagin numbers equals a linear combination of Hodge numbers, 
then it can be considered as a homomorphism $\varphi$ on $\Ho_{2n}$. The domain is spanned by products of $E$, $L$ 
and $S$, but any product with a complex curve as a factor has trivial Pontryagin numbers. Thus $\varphi$ factors through the 
projection $\Z[L,E,S]\longrightarrow\Z[S]$, which we can identify with the signature homomorphism, since the signature of $S$ is $\pm 1$. 
Conversely, the signature is a linear combination of Pontryagin numbers by the classical results of Thom.
\end{proof}

Returning to the generators $A$, $B$ and $C$ for $\Ho_*$ we can prove the following result, which implies
Theorem~\ref{t:Hbirat} stated in the introduction.
\begin{thm}\label{t:biratideal}
Let $\I\subset\Ho_*$ be the ideal generated by differences of birational smooth complex projective varieties. Then 
$\I = (C) = \ker (b)$, where  $C=xy\cdot z^2$ and $b\colon\Ho_*\longrightarrow\Z[y,z]$ is given by setting $x=0$
in the Hodge polynomials.
\end{thm}
\begin{proof}
If $S$ is a K\"ahler surface and $\hat S$ its blowup at a point, then $\hat S - S = C$, and so $(C)\subset\I$.

The homomorphism $b$ sends the Hodge polynomial in degree $n$ to $(h^{0,0}+h^{0,1}y+\ldots+h^{0,n}y^n)z^n$.
As the $h^{0,q}$ are birational invariants, cf.~\cite[p.~494]{GH}, we have $\I\subset\ker (b)$. From the proof of
Theorem~\ref{t:structure} we know already that there are no universal relations between the Hodge numbers, 
other than the ones generated by the K\"ahler symmetries, and so the image of $b$ in degree $n$ is a 
free $\Z$-module of rank $n+1$. Since $(C)\subset\ker (b)$, this means that $b$ maps $\Z[A,B]$ isomorphically
onto $\im (b)$, and so $(C)=\ker (b)$.
\end{proof}

This Theorem tells us exactly which linear combinations of Hodge numbers are birational invariants of projective 
varieties, or of compact K\"ahler manifolds. Indeed, any homomorphism $\varphi\colon\Ho_n\longrightarrow M$ of 
$\Z$-modules that vanishes on $\I\cap\Ho_n$ factors through $b$. This proves Theorem~\ref{t:Hbirat} stated in 
the introduction.

We already mentioned the homomorphism $\sigma\colon\Ho_*\longrightarrow\Z[z]$ given by the signature. It is a 
specialization (for $y=1$) of the Hirzebruch genus
$$
\chi\colon\Ho_*\longrightarrow\Z[y,z]
$$
defined by setting $x=-1$ in the Hodge polynomials. Consider a polynomial 
$$
(\chi_0+\chi_1 y+\ldots+\chi_n y^n   )\cdot z^n\in\im(\chi) \ .
$$
By Serre duality in $\Ho_n$, this must satisfy the constraint $\chi_p=(-1)^n\chi_{n-p}$. Let $\Hi_n$ be the 
$\Z$-module of all polynomials of the form $(\chi_0+\chi_1 y+\ldots+\chi_n y^n)z^n\in\Z[y,z]$ satisfying this constraint. 
It is clear that this is a free $\Z$-module of rank $[(n+2)/2]$, and that 
$$
\Hi_* =  \bigoplus_{n=0}^{\infty} \Hi_n\subset\Z[y,z] \ ,
$$
is a graded commutative ring. 
\begin{thm}\label{t:Hirz}
The Hirzebruch genus defines a surjective homomorphism $\chi\colon\Ho_*\longrightarrow\Hi_*$ of graded rings, whose kernel is the 
principal ideal in $\Ho_*$ generated by an elliptic curve. In particular $\Hi_*$ is a polynomial ring over $\Z$ with one 
generator in degree $1$ and one in degree $2$. As generators one may choose $\C P^1$ and $\C P^2$.
\end{thm}
\begin{proof}
It is clear that $\chi$ is a homomorphism of graded rings, and that elliptic curves are in its kernel. Identifying $\Ho_*$ with 
$\Z[E,\C P^1, \C P^2]$, the Hirzebruch genus factors through the projection $\Z[E,\C P^1, \C P^2]\longrightarrow\Z[\C P^1, \C P^2]$,
and we have to show that the induced homomorphism $\Z[\C P^1, \C P^2]\longrightarrow\Hi_*$ is an isomorphism. 
This follows from the proof of Theorem~\ref{t:structure}, where we showed that there are no unexpected relations between the Hodge 
numbers. In particular, there are no non-trivial relations between the coefficients $\chi_0, \chi_1, \ldots , \chi_{[n/2]}$.
Alternatively one can show that $\Z[\C P^1, \C P^2]\longrightarrow\Hi_*$ is an isomorphism by elementary manipulations 
using $\chi (\C P^1)=(1-y)z$ and $\chi (\C P^2)=(1-y+y^2)z^2$. 
\end{proof}

\begin{rem}
With coefficients in $\Q$, it is well known that the image of the Hirzebruch genus is a polynomial ring on the images of $\C P^1$ and $\C P^2$. 
That this also holds over $\Z$ was recently made explicit in~\cite[Remark~7.1]{S}. There, as everywhere
in the literature, the Hirzebruch genus is identified with the Todd genus on the complex bordism ring using the Hirzebruch--Riemann--Roch
theorem. However, by its very definition, it should be considered on the Hodge ring instead, which is a much simpler object than the 
bordism ring, and in particular is finitely generated. By HRR, the two interpretations give the same image, since the bordism ring is 
generated, over $\Z$, by K\"ahler manifolds, compare Subsection~\ref{ss:bordism} below.
\end{rem}

Theorem~\ref{t:Hirz} tells that there are no universal relations between the Hodge and Chern numbers other than the 
Hirzebruch--Riemann--Roch relations:
\begin{cor}\label{c:univ}
The mod $m$ reduction of a $\Z$-linear combination of Hodge numbers of smooth complex projective varieties equals a 
linear combination of Chern numbers if and only if, mod $m$ and modulo K\"ahler symmetries, it is a linear combination of the $\chi_p$.
\end{cor}
\begin{proof}
Since products with an elliptic curve as a factor have trivial Chern numbers, any linear combination of Hodge numbers 
that equals a combination of Chern numbers must factor through the projection $\Ho_*\longrightarrow\Ho_*/(E)$.
By Theorem~\ref{t:Hirz}, this projection is the Hirzebruch genus $\chi$. Conversely, by the Hirzebruch--Riemann--Roch
theorem, the coefficients of $\chi$ are expressed as linear combinations of Chern numbers via the Todd polynomials.
\end{proof}

\section{The comparison map and the Poincar\'e ring of K\"ahler manifolds}\label{s:comp}

In this section we analyse the comparison map
\begin{alignat*}{1}
f\colon \Ho_* &\longrightarrow\Po_*\\
x \longmapsto t \ \ , \ \ 
y &\longmapsto t \ \ , \ \
z \longmapsto z^2 \ 
\end{alignat*}
given by forgetting the K\"ahler structure on elements of $\Ho_*$, thus specializing Hodge polynomials to Poincar\'e 
polynomials. This map doubles the degree, since the real dimension of a K\"ahler manifold is twice its complex dimension. 
Here are the main properties of this homomorphism:
\begin{prop}\label{t:comp}
{\textrm 1.} The image of $f$ consists of all elements of $\Po_*\subset\Z[t,z]$ of even degree, whose coefficients of odd 
powers of $t$ are even.

\noindent
{\textrm 2.} The kernel of $f$ is a principal ideal in $\Ho_*$ generated by the following homogeneous element $G$ of degree $2$:
$$
G = 4\C P^2-3L^2+E^2-2EL \ .
$$
\end{prop}
\begin{proof}
In Section~\ref{s:rings} we defined $\Ho_*$ to be generated by all formal Hodge polynomials 
$$
\left(\sum_{p,q=0}^n h^{p,q}\cdot x^py^q\right)\cdot z^n \ ,
$$ 
in $\Z [x,y,z]$, satisfying the K\"ahler symmetries $h^{p,q}=h^{q,p}=h^{n-p,n-q}$.
Serre duality $h^{p,q}=h^{n-p,n-q}$ implies Poincar\'e duality for the image under $f$, whereas the symmetry $h^{p,q}=h^{q,p}$ 
implies that the image has even odd-degree Betti numbers.
Finally, since $f$ doubles the degree, its image is concentrated in even degrees.
Conversely, it is straightforward to check that the elements $e^{2n}_{n}$, $e^{2n}_k$ with even $k<n$, and 
$2e^{2n}_k$ with odd $k<n$ of $\Po_{2n}$ in Lemma~\ref{l:basis} are images of formal Hodge polynomials.
This establishes the first part of the Proposition.

For the second part, we note that $G = 4\C P^2-3L^2+E^2-2EL$
has zero Betti numbers and therefore lies in the kernel of $f$.
Thus $f$ induces a homomorphism $\hat f\colon\Ho_*/(G)\longrightarrow\Po_*$. By the first part of the proposition
proved above, the image of $f$, equivalently $\hat f$, in degree $2n$ is a free $\Z$-module of rank $n+1$. By 
Corollary~\ref{c:Hodge}, the degree $n$ part of $\Ho_*/(G)$ is generated as a $\Z$-module by $n+1$
elements. Therefore $\hat f$ is injective, and an isomorphism onto $\im (f)\subset\Po_*$. 
\end{proof}

By the first part of this proposition, a basis for $\Hom (f(\Ho_n),\Z)$ is given by the even-degree Betti numbers
and the halves of the odd-degree Betti numbers, both up to the middle dimension only because of Poincar\'e duality.
In particular, the only non-trivial congruences satisfied by the Betti numbers of K\"ahler manifolds are the vanishing
mod $2$ of the odd-degree Betti numbers.

Proceeding as in the definition of the Poincar\'e ring $\Po_*$ of oriented manifolds in Section~\ref{s:Pring}, we 
define the Poincar\'e ring of K\"ahler manifolds. This ring is the image of the comparison map $f$ in $\Po_*$.
Thus, Proposition~\ref{t:comp} yields:
\begin{thm} \label{t:KaehlerPring}
The Poincar\'e ring of K\"ahler manifolds $\im (f)$ is isomorphic to
$$
\Z[L,E,\C P^2]/(4\C P^2-3L^2+E^2-2EL) \ ,
$$
where $L=\C P^1$ is the projective line and $E$ an elliptic curve.
\end{thm}
Using this theorem, we can determine all universal relations between Betti and Pontryagin numbers of K\"ahler manifolds.
Since in odd complex dimensions there are no non-trivial Pontryagin numbers, we can restrict ourselves to even complex 
dimensions. In these dimensions, for K\"ahler manifolds only, Corollary~\ref{c:BPoriented} is strengthened as follows:
\begin{cor}\label{c:BPKaehler}
Any non-trivial congruence between an integral linear combination of Betti numbers of K\"ahler manifolds of even complex 
dimension $2n$ and an integral linear combination of Pontryagin numbers is a multiple of the following congruence between 
the Euler characteristic and the signature:
\begin{equation}\label{eq:sign}
e \equiv (-1)^n\sigma \ \ \mod 4 \ .
\end{equation}
\end{cor}
The word non-trivial in the formulation is meant to indicate that we ignore congruences where both sides vanish separately.
This is necessary because the odd-degree Betti numbers are all even.
\begin{proof}
The signature is a linear combination of Pontryagin numbers by the work of Thom. That it satisfies the 
congruence~\eqref{eq:sign} for compact K\"ahler manifolds follows from the Hodge index theorem.

Conversely, suppose we have a $\Z$-linear combination of Betti numbers that, on all K\"ahler manifolds of complex dimension 
$2n$, is congruent to a linear combination of Pontryagin numbers modulo $m$, but does not vanish identically mod $m$. 
Such a linear combination corresponds to a homomorphism $\varphi$ from the degree $4n$ part of the Poincar\'e 
ring of K\"ahler manifolds to $\Z_m$ that vanishes on all elements with zero Pontryagin numbers.
Since the Pontryagin numbers vanish on manifolds that are products with a complex curve as a factor, 
Theorem~\ref{t:KaehlerPring} shows that $\varphi$ factors through the degree $4n$ part of $\Z[\C P^2]/(4\C P^2)$.
Now the mod $4$ reduction of the Euler characteristic gives an isomorphism between $\Z[\C P^2]/(4\C P^2)$ and
$\Z_4[z^4]$. This completes the proof.
\end{proof}

Replacing the Pontryagin numbers by the Chern numbers of K\"ahler manifolds, we obtain the following:
\begin{cor}\label{c:BC}
A $\Z$-linear combination of Betti numbers of K\"ahler manifolds is congruent mod $m$ to a non-trivial 
linear combination of Chern numbers if and only if, mod $m$, it is a multiple of the Euler characteristic.
\end{cor}
Again we do not consider congruences where the two sides vanish separately.
\begin{proof}
Since the Euler characteristic of a K\"ahler manifold equals the top Chern number $c_n$, one direction is clear. For the converse,
assume that, in complex dimension $n$, the mod $m$ reduction of some $\Z$-linear combination of Chern numbers equals a linear 
combination of Betti numbers. This corresponds to a non-trivial homomorphism from the degree $2n$ part of the Poincar\'e ring of K\"ahler 
manifolds to $\Z_m$. Since any product with an elliptic curve has trivial Chern numbers, Theorem~\ref{t:KaehlerPring} shows that 
this homomorphism descends to a homomorphism from the degree $2n$ part of $\Z[L,\C P^2]/(4\C P^2-3L^2)$ to $\Z_m$.
Upon identifying this ring with the subring of $\Z [z^2]$ generated by $2z^2$ and $z^4$, the projection from the Poincar\'e ring of K\"ahler 
manifolds to $\Z[L,\C P^2]/(4\C P^2-3L^2)$ is identified with the Euler characteristic, obtained by setting $t=-1$ in the 
Poincar\'e polynomials. This completes the proof.
\end{proof}

\section{The Hirzebruch problem for Hodge numbers}\label{s:HH}

In this section we solve Hirzebruch's problem concerning Hodge numbers by proving Theorem~\ref{t:main2} 
stated in the introduction.
The following is the first step in its proof.

\begin{thm}\label{t:ker}
The ideal in the Hodge ring $\Ho_*$ generated by the differences of homeomorphic smooth complex
projective varieties coincides with the kernel of the forgetful map $f\colon\Ho_*\longrightarrow\Po_*$.
\end{thm}
\begin{proof}
Let $\I\subset\Ho_*$ be the ideal generated by 
$$
\{ M-N \ \vert \ M, \ N \   \textrm{homeomorphic} \ \textrm{projective} \ \textrm{varieties}\ \textrm{of} \ \textrm{dimension} \ n \} \ ,
$$
for all $n$.
These are differences of smooth complex projective varieties of complex dimension $n$ that are homeomorphic,
without any assumption about compatibility of their orientations under homeomorphisms. 

Since Poincar\'e polynomials are homeomorphism invariants, it is clear that $\I\subset\ker (f)$. To prove $\ker (f)\subset\I$ we use
Proposition~\ref{t:comp}, telling us that $\ker (f)$ is a principal ideal generated by an element $G$ in degree $2$. This $G$ has the property
that all its Betti numbers vanish, and its signature equals $+4$. 
We only have to prove that $G\in\I$.

By the results of~\cite{MAorient} there are many pairs $(X,Y)$ of simply connected projective surfaces of non-zero signature
that are orientation-reversingly homeomorphic with respect to the orientations given by the complex structures. The only divisibility
condition that has to be satisfied in all cases is that the signatures must be even. More specifically, by~\cite[Theorem~3.7]{MAorient},
we can choose two such pairs $(X_1,Y_1)$ and $(X_2,Y_2)$ with the property that the greatest common divisor of the
signatures $\sigma (X_1)$ and $\sigma (X_2)$ is $2$. Then there are integers $a$ and $b$ such that 
\begin{equation}\label{eq:ab}
a\sigma (X_1)+b\sigma (X_2)=2 \ .
\end{equation}
We now claim that the following identity holds:
\begin{equation}\label{eq:id}
H_{x,y,z}(G) = a(H_{x,y,z}(X_1)-H_{x,y,z}(Y_1))+b(H_{x,y,z}(X_2)-H_{x,y,z}(Y_2)) \ .
\end{equation}
Since $X_i-Y_i\in\I$, this proves that $G\in\I$.

To prove~\eqref{eq:id} note that the Betti numbers vanish on both the left-hand and the right-hand sides. Therefore, to check that 
all Hodge numbers agree, we only have to check the equality of the signatures, as follows:
$$
\sigma (a(X_1-Y_1)+b(X_2-Y_2)) = 2\sigma (aX_1+bX_2)=4=\sigma (G) \ ,
$$
where the first equality comes from the fact that $X_i$ and $Y_i$ are orientation-reversingly homeomorphic and the second 
equality comes from~\eqref{eq:ab}. This completes the proof of the theorem.
\end{proof}

Next we consider differences of diffeomorphic, not just homeomorphic, projective varieties.
\begin{thm}\label{t:Hdiff}
In degrees $n\geq 3$ the kernel of $f\colon\Ho_n\longrightarrow\Po_{2n}$ is generated as a $\Z$-module by differences of 
diffeomorphic smooth complex projective varieties.

In all degrees the intersection $\ker (f)\cap\ker (\sigma )$  is generated as a $\Z$-module by differences of 
smooth complex projective varieties that are orientation-preservingly diffeomorphic with respect to
the orientations induced by the complex structures.
\end{thm}
\begin{proof}
By the proof of Theorem~\ref{t:ker}, the ideal $\ker (f)$ is generated by differences of pairs of homeo\-morphic 
simply connected algebraic surfaces $(X_i,Y_i)$. Identifying $\Ho_*$ with $\Z[E,\C P^1,\C P^2]$, we see that the kernel of 
$f\colon\Ho_n\longrightarrow\Po_{2n}$ is generated as a $\Z$-module by products of the $X_i-Y_i$ with $E$, $\C P^1$ and $\C P^2$.

By a result of Wall~\cite{W}, the smooth four-manifolds $X_i$ and $Y_i$ are smoothly $h$-cobordant. It follows that $X_i\times \C P^j$ 
and $Y_i\times\C P^j$ are also $h$-cobordant, and are therefore diffeomorphic by Smale's $h$-cobordism theorem~\cite{Smale}. 
Products of $X_i-Y_i$ with powers of $E$, therefore not involving a $\C P^j$, are handled by the following Lemma, which is a 
well-known consequence of the $s$-cobordism theorem of Barden, Mazur and Stallings; see~\cite[p.~41/42]{Ker}:
\begin{lem}\label{l:scob}
Let $M$ and $N$ be $h$-cobordant manifolds of dimension $\geq 5$. Then $M\times S^1$ and $N\times S^1$ are diffeomorphic.
\end{lem}
This shows that the products $X_i\times E$ and $Y_i\times E$ are diffeomorphic, completing the proof of the first statement.

For the second statement note that $\ker (f)\cap\ker (\sigma )$ vanishes in degrees $< 3$. Therefore we only have to
consider the degrees already considered in the first part. The generators considered there all have zero signature, 
except the products of $X_i-Y_i$ with pure powers of $\C P^2$. This implies that the products of $X_i-Y_i$ with monomials 
in $E$, $\C P^1$ and $\C P^2$ that involve at least one of the curves generate $\ker (f)\cap\ker (\sigma )$. Since $E$ and 
$\C P^1$ admit orientation-reversing self-diffeomorphisms,
it follows that $X_i\times E$ and $Y_i\times E$, respectively $X_i\times\C P^1$ and $Y_i\times\C P^1$, are not just
diffeomorphic, as proved above, but that the diffeomorphism may be chosen to preserve the orientations.
This completes the proof.
\end{proof}

We can now give a complete answer to Hirzebruch's question concerning Hodge numbers.

\begin{proof}[Proof of Theorem~\ref{t:main2}]
We consider integral linear combinations of Hodge numbers as homomorphisms $\varphi\colon\Ho_n\longrightarrow\Z_m$.
If a linear combination of Hodge numbers defines an unoriented homeomorphism invariant, then by Theorem~\ref{t:ker}
the corresponding homomorphism $\varphi$ factors through $f$. Looking at the description of $\im (f)$ in Proposition~\ref{t:comp},
we see that every homeomorphism-invariant linear combination of Hodge numbers is a combination of the even-degree
Betti numbers and the halves of the odd-degree Betti numbers. By the first part of Theorem~\ref{t:Hdiff}, the same 
conclusion holds for unoriented diffeomorphism invariants in dimensions $n\neq 2$.

Combining the above discussion with the second part of Theorem~\ref{t:Hdiff} completes the proof of Theorem~\ref{t:main2}.
\end{proof}

\begin{ex}
By the results of~\cite{MAorient} used above, the signature itself is not a homeomorphism invariant of smooth complex projective varieties.
However, the reduction mod $4$ of the signature is a homeomorphism invariant, since by the proof of Theorem~\ref{t:ker}, it vanishes on 
the ideal $\I=\ker (f)$. Theorem~\ref{t:main2} then tells us that the signature of a K\"ahler manifold is congruent
mod $4$ to a linear combination of even-degree Betti numbers and halves of odd-degree Betti numbers. 
This latter fact also follows from the Hodge index theorem, which gives the precise congruence~\eqref{eq:sign}.
\end{ex}

\section{The Chern--Hodge ring}\label{s:CH}

\subsection{Unitary bordism}\label{ss:bordism}

We now recall the classical results about the complex bordism ring $\Om_* = \oplus_{n=0}^{\infty}\Om_n$ that we shall need.
By results of Milnor~\cite{M,TM} and Novikov~\cite{NN} this is a polynomial ring over $\Z$ on countably many generators $\beta_i$, one for 
every complex dimension $i$. In particular, the degree $n$ part $\Om_n$ is a free $\Z$-module of rank $p(n)$, the number 
of partitions of $n$. Two stably almost complex manifolds of the same dimension have the same Chern numbers if and only if they represent the 
same element in $\Om_*$.

The $\beta_i$ are commonly referred to as a basis sequence, and we will need to discuss some
special choices of such basis sequences. An element $\beta_n\in\Om_n$ can be taken as a generator over $\Z$ if and only if 
a certain linear combination of Chern numbers $s_n$, referred to as the Thom-Milnor number, satisfies 
$s_n(\beta_n)=\pm 1$ if $n+1$ is not a prime power, and $s_n(\beta_n)=\pm p$ if $n+1$ is a power of the prime $p$. 

In the case of $\Om_*\otimes\Q$ one may take $\beta_i = \C P^i$ as a basis sequence, but this is not a basis sequence
over $\Z$. Milnor proved that one can obtain a basis sequence over $\Z$ by considering formal $\Z$-linear combinations of 
complex projective spaces and of smooth hypersurfaces $H\subset \C P^k\times\C P^{i+1-k}$ of bidegree $(1,1)$, cf.~\cite{TM}
and~\cite[pp.~249--252]{MDT}. It follows
that one may take (disconnected) projective, in particular K\"ahler, manifolds for the generators of $\Om_*$ over $\Z$. 
These projective manifolds are very special, in that they are birational to $\C P^i$:
\begin{lem}\label{l:Mrat}
Milnor manifolds, that is, smooth hypersurfaces $H\subset \C P^k\times\C P^{i+1-k}$ of bidegree $(1,1)$, are rational.
\end{lem}
\begin{proof}
Let $x$ and $y$ be homogeneous coordinates on $\C P^k$ respectively $\C P^{i+1-k}$. In an affine chart $\C^k\times\C^{i+1-k}=\C^{i+1}$
given by $x_0\neq 0 \neq y_0$, say, the defining equation of $H$ of bidegree $(1,1)$ in $x$ and $y$ becomes a quadratic 
equation in the coordinates of $\C^{i+1}$. Therefore $H$ is birational to an irreducible quadric in $\C P^{i+1}$, which is well known 
to be rational.
\end{proof}

Finally the Todd genus $\Td\colon\Om_*\longrightarrow\Hi_*$ is the ring homomorphism sending a bordism class $[M]$ to
$(\Td_0(M)+\Td_1(M) y+\ldots+\Td_n(M) y^n)z^n$, where the $\Td_p$ are certain combinations of Chern numbers. 
By the Hirzebruch--Riemann--Roch theorem one has $\Td_p=\chi_p = \sum_q(-1)^q h^{p,q}$.

\subsection{Combining the Hodge and bordism rings}

We now consider finite linear combinations of equidimensional compact K\"ahler manifolds with coefficients in $\Z$,
and identify two such linear combinations if they have the same dimensions and the same Hodge and Chern numbers.
The set of equivalence classes is naturally a graded ring, graded by the dimension, with multiplication induced by the 
Cartesian product of K\"ahler manifolds. We call this the Chern--Hodge ring $\CH_*$.

The degree $n$ part $\CH_n$ of the Chern--Hodge ring is the diagonal submodule $\Delta_n\subset\Ho_n\oplus\Om_n$ 
generated by all
$$
(H_{x,y,z}(M^n),[M^n])\in\Ho_n\oplus\Om_n \ , 
$$
where $M$ runs over compact K\"ahler manifolds of complex dimension $n$ and the square brackets denote bordism classes.

\begin{prop}\label{p:diagonal}
The diagonal submodule $\Delta_n$ is the kernel of the surjective homomorphism 
\begin{alignat*}{1}
h\colon \Ho_n\oplus\Om_n &\longrightarrow\Hi_n\\
(H_{x,y,z}(M),[N]) &\longmapsto \chi (M)-\Td(N) \ ,
\end{alignat*}
where $\chi\colon\Ho_*\longrightarrow\Hi_*$ is the Hirzebruch genus, and $\Td\colon\Om_*\longrightarrow\Hi_*$ is the 
Todd genus.
\end{prop}
\begin{proof}
The surjectivity of $h$ follows from the surjectivity of $\chi$ proved in Theorem~\ref{t:Hirz}.

By the Hirzebruch--Riemann--Roch theorem $\Delta_n\subset\ker (h)$. 
To check the reverse inclusion consider an element $(H_{x,y,z}(M),[N])\in\ker (h)$.
This means $\chi(M)=\Td(N)$, and so, applying HRR to $N$, $\chi(M)=\chi(N)$. Since by Theorem~\ref{t:Hirz} the kernel 
of $\chi$ is the principal ideal generated by an elliptic curve $E$, we conclude that in the Hodge ring the difference of 
$M$ and $N$ is of the form $E\cdot P$, where $P$ is a homogeneous 
polynomial of degree $n-1$ in the generators of $\Ho_*$. Thus in $\Ho_n\oplus\Om_n$ we may write
$$
(H_{x,y,z}(M),[N]) = (H_{x,y,z}(N),[N]) + (H_{x,y,z}(E\cdot P), 0) \ .
$$
Since an elliptic curve $E$ represents zero in the bordism ring, we have $(H_{x,y,z}(E\cdot P), 0)=(H_{x,y,z}(E\cdot P), [E\cdot P])$,
and so the second summand on the right hand side is in the diagonal submodule. As the first summand is trivially in $\Delta_n$,
we have now proved $\ker (h)\subset\Delta_n$.
\end{proof}

As a consequence of Proposition~\ref{p:diagonal}, $\CH_n=\Delta_n$ is a free $\Z$-module of rank
\begin{alignat*}{1}
\rk\CH_n &= \rk\Ho_n+\rk\Om_n-\rk\Hi_n \\
&= \big[\frac{n+2}{2}\big]\cdot\big[\frac{n+3}{2}\big] + p(n)-\big[\frac{n+2}{2}\big] \\
&=  \big[\frac{n+2}{2}\big]\cdot\big[\frac{n+1}{2}\big] + p(n) \ .
\end{alignat*}

The structure of the Chern--Hodge ring is described by the following result.
\begin{thm}\label{t:CHring}
Let $\beta_1=\C P^1, \beta_2, \beta_3,\ldots$ be $\Z$-linear combinations of K\"ahler manifolds forming a basis sequence for the 
complex bordism ring $\Om_*$, and let $P_i(E,\beta_1,\beta_2)$ be the unique polynomial in $E$, $\beta_1$ and 
$\beta_2$ having the same image in the Hodge ring as $\beta_i$. Then the Chern--Hodge ring $\CH_*$ is isomorphic as a graded 
ring to the quotient of $\Z[E,\beta_1,\beta_2,\beta_3,\ldots]$ by the ideal $\I$ generated by all $E(\beta_i-P_i(E,\beta_1,\beta_2))$.
\end{thm}
\begin{proof}
In degree $2$ the Thom-Milnor number $s_2$ of a K\"ahler surface equals $c_1^2-2c_2$, which is $3$ times the signature.
Since $\beta_2$ is a generator of the bordism ring, we have $s_2(\beta_2)=\pm 3$, so $\beta_2$ has signature $\pm 1$.
By Corollary~\ref{c:Hodge} this means that $\Ho_*=\Z[E,\beta_1,\beta_2]$. Therefore, for each $\beta_i$ there is 
indeed a unique polynomial $P_i(E,\beta_1,\beta_2)$ having the same image as $\beta_i$ in $\Ho_*$.

Consider the canonical ring homomorphism
$$
\phi\colon\Z[E,\beta_1,\beta_2,\beta_3,\ldots]\longrightarrow\CH_* \ .
$$
We first prove that $\phi$ is surjective. Let $M$ be a compact K\"ahler manifold of dimension $n$, and $[M]\in\Om_n$
its bordism class. We need to show that $(H_{x,y,z}(M),[M])\in\im (\phi )$.
Since the $\beta_i$ form a basis sequence for the bordism ring, there is a unique homogeneous polynomial 
$P$ of degree $n$ in the $\beta_i$ such that $[M]=[P(\beta_i,\ldots,\beta_n)]\in\Om_n$. We then have 
$\phi (P)= (H_{x,y,z}(P),[M])\in\CH_n$. Moreover, $H_{x,y,z}(P)-H_{x,y,z}(M)$ is in the kernel of the Hirzebruch
genus, which by Theorem~\ref{t:Hirz} is the ideal $(E)\subset\Ho_*$. Thus, in $\Ho_*$ we may write $M=P+EQ$,
where $Q$ is a homogeneous polynomial of degree $n-1$ in $E$, $\beta_1$ and $\beta_2$. Since $E$ maps to
zero in the bordism ring, we conclude $\phi (P+EQ)=(H_{x,y,z}(M),[M])$. This completes the proof of surjectivity.

Finally we need to show that $\ker (\phi )=\I$. By the definition of $\I$, we have $\I\subset\ker (\phi )$, and so 
$\phi$ descends to the quotient $\Z[E,\beta_1,\beta_2,\beta_3,\ldots]/\I$. The degree $n$ part of this 
quotient surjects to $\CH_n$, which is a free module of rank 
$$
\big[\frac{n+2}{2}\big]\cdot\big[\frac{n+1}{2}\big] + p(n) \ ,
$$
where $p(n)=\rk\Om_n$ is the number of partitions of $n$. Looking at the definition of $\I$ we see that the degree 
$n$ part of the quotient $\Z[E,\beta_1,\beta_2,\beta_3,\ldots]/\I$ is generated as a $\Z$-module by $\rk\CH_n$ many
monomials. Since we know already that $\phi$ is surjective, this shows that $\phi$ is injective, and therefore an
isomorphism.
\end{proof}

We can now generalize Theorem~\ref{t:biratideal} from the Hodge to the Chern--Hodge ring:
\begin{thm}\label{t:CHbiratideal}
Let $\I\subset\CH_*$ be the ideal generated by differences of birational smooth complex projective varieties. Then 
there is a basis sequence for the bordism ring with $\beta_1=\C P^1$ and $\beta_i\in\I$ for all $i\geq 2$.
Furthermore, $\I$ is the kernel of the composition
$$
\CH_*\stackrel{p}{\longrightarrow}\Ho_*\stackrel{b}{\longrightarrow}\Z[y,z] \ ,
$$
where  $p\colon\CH_*\longrightarrow\Ho_*$ is the projection  
and $b\colon\Ho_*\longrightarrow\Z[y,z]$ is given by setting $x=0$ in the Hodge polynomials.
\end{thm}
\begin{proof}
Take $\beta_1=\C P^1$, and $\C P^2 - \C P^1\times \C P^1 = -C$ as the generator $\beta_2$ in degree $2$. 
In higher degrees we take the Milnor generators, which are formal linear combinations of projective spaces and 
of Milnor manifolds, and, like in degree $2$, subtract from each projective space or Milnor manifold a copy of 
$\beta_1^n=\C P^1\times\ldots\times\C P^1$. This does not change the property of being generators (over $\Z$), 
but, after this subtraction, we have generators $\beta_i$ which for $i\geq 2$ are contained in $\I$ by Lemma~\ref{l:Mrat}.
This completes the proof of the first statement. For the second statement note that by Theorem~\ref{t:biratideal}
the ideal $\I$ is contained in the kernel of $b\circ p$. Conversely, our choice of generators shows that $\ker (b\circ p)\subset\I$.
\end{proof}
As a consequence of this result, Theorem~\ref{t:Hbirat} holds for combinations of Hodge and Chern numbers:
\begin{cor}\label{c:CHbirat}
The mod $m$ reduction of an integral linear combination of Hodge and Chern numbers is a birational invariant of smooth complex 
projective varieties if and only if after adding a suitable combination of the $\chi_p-\Td_p$ it is congruent to a linear combination of 
the $h^{0,q}$ plus a linear combination of Chern numbers that vanishes mod $m$ when evaluated on any smooth complex projective variety.
\end{cor}
One should keep in mind that the Hodge numbers in this statement are, as always, taken modulo the K\"ahler symmetries.
The corresponding statement over $\Q$ follows from the statement about congruences.

\section{The general Hirzebruch problem}\label{s:Hgen}

Finally we address the general version of Hirzebruch's Problem~31 from~\cite{Hir1} asking which linear combinations of Hodge and Chern numbers
are topological invariants. This combines the work about Hodge numbers in Section~\ref{s:HH} above with the work on Chern numbers
in~\cite{Chern}. The first step is the following result.
\begin{thm}\label{t:CHtop1}
The ideal $\J$ in $\CH_*\otimes\Q$ generated by differences of homeomorphic projective varieties is the kernel of 
the forgetful homomorphism
$$
F\colon\CH_*\otimes\Q\longrightarrow\Po_*\otimes\Q \ .
$$

In degrees $\geq 3$ this ideal coincides with the one generated by differences of diffeomorphic projective varieties.
\end{thm}
\begin{proof}
Since Poincar\'e polynomials are homeomorphism invariants, it is clear that $\J\subset\ker (F)$.

By~\cite[Theorem~10]{Chern} there is a basis sequence $\beta_1=\C P^1,\beta_2,\beta_3,\ldots$ for 
$\Om_*\otimes\Q$ with $\beta_i\in\J$ for all $i\geq 2$. On the one hand, this means that, in the description of 
$\CH_*\otimes\Q$ as a quotient of the polynomial ring $\Q[E,\beta_1,\beta_2,\beta_3,\ldots]$ given
by Theorem~\ref{t:CHring}, the only monomials in the generators whose residue classes are not 
necessarily in $\J$ are those involving only $E$ and $\beta_1$. On the other hand, it is clear
from Corollary~\ref{c:Ptensor} that the residue class of a non-trivial polynomial in $E$ and $\beta_1$ 
cannot be in $\ker (F)$. Thus $\ker (F)$ is the ideal generated by the 
$\beta_i$ with $i\geq 2$, and is therefore contained in $\J$. This proves the first statement in the theorem.

For the second statement note that the $\beta_i$ used above are in fact differences of diffeomorphic 
projective varieties as soon as $i\geq 3$, see~\cite[Theorem~9]{Chern}, and that the same is true for 
$\beta_1\cdot\beta_2$ and $\beta_2\cdot\beta_2$. The generator $\beta_2$ is a difference of 
orientation-reversingly homeomorphic simply connected algebraic surfaces $X$ and $Y$. As in the 
proof of Theorem~\ref{t:Hdiff} above it follows from Lemma~\ref{l:scob} that $E\times X$ and $E\times Y$
are diffeomorphic, and so $E\cdot\beta_2$ is also a difference of diffeomorphic projective varieties.
\end{proof}

Next we look at oriented topological invariants. For this it is convenient to introduce the oriented analogue 
of the Chern--Hodge ring. Consider formal $\Z$-linear combinations of equidimensional closed oriented 
smooth manifolds, and identify two such combinations if they have the same dimension, the same Betti
numbers, and the same Pontryagin numbers. The quotient is again a graded ring, which we call the Pontryagin--Poincar\'e
ring $\PP_*$, graded by the dimension. By Corollary~\ref{c:BPoriented} there are no $\Q$-linear relations between the Betti
and Pontryagin numbers. Therefore, by the classical result of Thom on the oriented bordism ring, we conclude 
$$
\PP_*\otimes\Q=(\Po_*\otimes\Q)\oplus(\Omega_*^{SO}\otimes\Q) \ ,
$$
where $\Omega_*^{SO}$ denotes the oriented bordism ring.

\begin{thm}\label{t:CHtop2}
The forgetful homomorphism 
$$
\tilde F\colon\CH_*\otimes\Q\longrightarrow\PP_*\otimes\Q
$$
is surjective onto the even-degree part of $\PP_*\otimes\Q$. Its kernel is the ideal $\JJ$ in $\CH_*\otimes\Q$ 
generated by differences of orientation-preservingly diffeomorphic smooth complex projective varieties.
\end{thm}
\begin{proof}
Let $E$ be an elliptic curve and $\beta_1=\C P^1$, both considered as elements in $\CH_*\otimes\Q$.
By Corollary~\ref{c:Ptensor} all $\Q$-linear combinations of Betti numbers in even dimensions are detected by polynomials 
in $\tilde F(E)$ and $\tilde F (\beta_1)$. These elements in $\PP_*\otimes\Q$ have trivial Pontryagin numbers.
Using the same basis sequence $\beta_i$ as in the previous proof, we see that the images $\tilde F (\beta_i)$
have trivial Betti numbers if $i\geq 2$, but any non-trivial linear combination of Pontryagin numbers is 
detected by polynomials in the $\tilde F(\beta_i)$ with even $i$. This proves that the image of $\tilde F$ is the even-degree part of 
$\PP_*\otimes\Q$.
 
It is clear that $\JJ\subset\ker (\tilde F)$ since Betti and Pontryagin numbers are oriented diffeomorphism invariants.
By definition, $\JJ$ is a subideal of $\J$, which, by the previous theorem, equals $\ker (F)$.

Using the same basis sequence as in the previous proof, $\J=\ker (F)$ is the ideal generated by all $\beta_i$
with $i\geq 2$. By~\cite[Theorem~7]{Chern}, this basis sequence has the property that for odd $i\geq 3$ the 
elements $\beta_i$ and $\beta_1\cdot\beta_{i-1}$ are in $\JJ$. We also know that, for all $i\geq 2$, $E\cdot\beta_i$ is a 
difference of diffeomorphic projective varieties. Since $E$ admits orientation-reversing self-diffeomorphisms
we have $E\cdot\beta_i\in\JJ$.

By the proof of surjectivity of $\tilde F$ onto the even-degree part of $\PP_*\otimes\Q$,
 no non-trivial polynomial in the $\beta_i$ with $i$ even can be in $\ker (\tilde F)$. Thus 
$\ker (\tilde F)$ is the ideal generated by the residue classes of the 
$\beta_i$ with odd $i\geq 3$, and by the $\beta_1\cdot\beta_j$ and $E\cdot\beta_j$ with $j$ even.
All these generators are in $\JJ$, and so $\ker (\tilde F)\subset\JJ$. This completes
the proof.
\end{proof}

\begin{rem}\label{r:cong}
In Theorems~\ref{t:CHtop1} and~\ref{t:CHtop2} we worked over $\Q$ in order to be able to use the special basis 
sequences $\beta_n$ for the unitary bordism ring constructed in~\cite{Chern}. For $n\geq 5$ we could use instead
certain generators for $\Om_*\otimes\Z[\frac{1}{2}]$ constructed in~\cite[Prop.~4.1]{S}. A generator $\beta_2$ with all the required
properties, that would also work after inverting only $2$, was obtained in the proof of Theorem~\ref{t:ker} above,
but in degrees $n=3$ or $4$ we do 
not have any alternative generators. Checking the numerical factors in~\cite[Prop.~15]{Chern}, it turns out that 
the $\beta_3$ used in~\cite{Chern} and in the above proofs works for $\Om_*\otimes\Z[\frac{1}{2}]$, but 
the $\beta_4$ used there requires one to invert $3$, in addition to inverting $2$. Therefore, Theorems~\ref{t:CHtop1} 
and~\ref{t:CHtop2} are true for $\CH_*\otimes\Z[\frac{1}{6}]$.
\end{rem}

We can finally prove Theorem~\ref{t:main3}.

\begin{proof}[Proof of Theorem~\ref{t:main3}]
The vector space dual to $\CH_n\otimes\Q$ is made up of $\Q$-linear combinations of Hodge and Chern numbers, modulo
the linear combinations of the $\chi_p-\Td_p$, and modulo the implicit K\"ahler symmetries. 
If a linear form on $\CH_n\otimes\Q$ defines an unoriented homeomorphism invariant, 
or an unoriented diffeomorphism invariant in dimension $n\geq 3$, then by Theorem~\ref{t:CHtop1}
the corresponding homomorphism $\varphi$ factors through $F$, and so reduces to a combination of Betti numbers. 
Conversely, linear combinations of Betti numbers are of course homeomorphism-invariant.
This completes the proof of the second statement.

By Theorem~\ref{t:CHtop2} a linear form on $\CH_n\otimes\Q$ that defines an oriented diffeomorphism invariant
factors through $\tilde F$, and therefore reduces to a combination of Betti and Pontryagin numbers,
which make up the linear forms on $\PP_{2n}\otimes\Q$. Conversely,
these linear combinations are invariant under orientation-preserving diffeomorphisms, and even under orientation-preserving
homeomorphisms by  a result of Novikov~\cite{N}. This completes the proof of the first statement.
\end{proof}

\bibliographystyle{amsplain}

\bigskip

\end{document}